\documentclass{amsart}

\usepackage{hyperref}
\usepackage{graphicx} 
\usepackage{amsmath}
\usepackage{amsthm}
\usepackage{amssymb}
\usepackage{textcomp, gensymb}
\usepackage{fullpage}
\usepackage{graphicx}
\usepackage{thm-restate}
\usepackage{tikz} 
\usepackage{booktabs, makecell}
\usepackage[mathscr]{euscript}
 \let\mathscr\relax
\usepackage[scr]{rsfso}
\usepackage[shortlabels]{enumitem}
\usepackage{comment}
\usepackage{parskip}
\usepackage{fancyhdr} 
\usepackage{upgreek}
\usepackage{tabularx}

\newcommand{\R}{\mathbb{R}}

\newcommand{\Z}{\mathbb{Z}}

\newcommand{\C}{\mathbb{C}}

\newcommand{\F}{\mathcal{F}}

\newcommand{\Todd}{\text{Todd}}

\newtheorem{thm}{Theorem}[section]
\newtheorem{cor}[thm]{Corollary}
\newtheorem{lem}[thm]{Lemma}
\newtheorem{prop}[thm]{Proposition}

\theoremstyle{definition}

\newtheorem{defin}[thm]{Definition}
\newtheorem{ex}[thm]{Example}

\newtheorem{que}[thm]{Question}

\newcommand{\ehr}{\text{ehr}}
\newcommand{\Ehr}{\text{Ehr}}

\usepackage{enumitem}
\setlist{nolistsep}

\usepackage{graphicx} 
\usepackage{amsmath}
\usepackage{amssymb}
\usepackage{mathtools}
\usepackage{arydshln}
\usepackage{comment}
\usepackage{graphicx} 
\usepackage{amsmath}
\usepackage{amssymb}
\usepackage{mathtools}
\usepackage{arydshln}
\usepackage{amsthm}
\usepackage{xcolor}

\theoremstyle{definition}

\usepackage{algorithm}
\usepackage{algpseudocode}

\title{Computing parametric weighted Ehrhart
polynomials of smooth polytopes}

\author{Daniel Hwang, Juliet Whidden, Josephine Yu}
\date{\today}

\begin{document}

\begin{abstract} 
We show that when integral polytopes are deformed while keeping the same facet normal vectors, the coefficients of weighted Ehrhart and $h^*$-polynomials are piecewise polynomial functions in the ``right hand sides'' of the linear inequalities defining the polytopes.  We give an algorithm and an implementation in SageMath for computing these polynomials for smooth polytopes, such as type $A$ alcoved polytopes, using a weighted Euler--Maclaurin type formula by Khovanskiǐ and Pukhlikov.  We  discuss some natural questions concerning signs of the coefficients of the weighted $h^*$-polynomials.

\smallskip
\noindent \textbf{Keywords.} Ehrhart Theory, lattice polytopes, alcoved polytopes, computations
\end{abstract}

\maketitle

\section{Introduction}
The central problem of Ehrhart theory is counting the integer points in integer dilates of an integral polytope~$P$. Ehrhart proved in 1967 that for any integral polytope $P$ of dimension $d$, the function $\ehr_{P}(t) := \#(tP \cap \mathbb{Z}^{d})$ is  a degree $d$ polynomial in $t$ \cite{ehrhart67}, which we call the {\bf Ehrhart polynomial}.  The \textbf{Ehrhart series }is a rational function of the form 
$$\Ehr_{P}(z) := \sum_{t \geq 0} \ehr_{P}(t) z^{t} = \frac{h_{d}^{*} z^{d} + \ldots + h_{0}^{*}}{(1 - z)^{d+1}},$$
where the numerator $h_{P}^{*}(z) = h_{d}^{*} z^{d} + \ldots + h_{0}^{*}$ is called the \textbf{Ehrhart $h^*$-polynomial.} Stanley proved in 1980 that the coefficients of $h^*_P$  are nonnegative integers \cite{STANLEY1980333}.  

Let $A$ be a fixed $n\times d$ integer matrix, and we will consider polytopes of the form 
\[
P_A(\mathbf{b}) := \{\mathbf{x} \in \R^d : A {\bf x} \leq \mathbf{b}\}
\]
for $\mathbf{b} = (b_1 \ldots b_n) \in \Z^n$.
 The integer point count $\#(P_A(\mathbf{b}) \cap \Z^d)$ is a piecewise quasipolynomial function in $b_1,\dots,b_n$, where the pieces correspond to combinatorial types of normal fans of polytopes of this form.  Sturmfels gave a precise description of these in~\cite{STURMFELS1995302}.  In the case when $A$ is unimodular, the polytope $P_A(\mathbf{b})$ is integral whenever $\mathbf{b}$ is integral, and the integer point counting function is a piecewise polynomial. De Loera and Sturmfels showed how to compute this efficiently~\cite{DeLoeraSturmfels}.

 Let $w$ be a homogeneous polynomial of degree $m$ in $d$ variables, which we use to count integer points in $\Z^d$ with weight $w$. Then the {\bf weighted Ehrhart function} $\ehr_{P, w}(t) = \sum_{\mathbf{p} \in tP \cap \mathbb{Z}^{d}} w(\mathbf{p})$ of an integral polytope $P$ is a polynomial of degree at most $d+m$, and the \textbf{weighted Ehrhart series} is a rational function of the form
$$\Ehr_{P, w}(z) = \sum_{t \geq 0} \ehr_{P, w}(t)z^{t} = \frac{h_{d + m}^{*}z^{d + m} + \ldots + h_{0}^{*}}{(1-z)^{d + m + 1}}.$$
The numerator $h_{P, w}^{*}(z) = h_{d + m}^{*}z^{d + m} + \ldots + h_{0}^{*}$ is called the \textbf{weighted $h^{*}$-polynomial} of $P$ with respect to weight $w$. The coefficients of  $h_{P, w}^{*}$ may be negative even when the weight is homogeneous and nonnegative on the polytope, but some sufficient conditions for positivity are known~\cite{Bajo_2024}.  
Our motivation comes from the following questions:
\begin{itemize} 
    \item How do the weighted Ehrhart and $h^*$-polynomials depend on $\mathbf{b}$?
    \item How can we efficiently compute the weighted Ehrhart and $h^*$-polynomials?
    \item Which combinations of integral polytopes $P$ and homogeneous polynomial weights $w$ give weighted $h^*$-polynomials with nonnegative coefficients?
\end{itemize}

We will answer the first two questions based on results of Pukhlikov and Khovanskiǐ \cite{Khovanskii-Pukhilkov-Virtual-Polytopes, Khovanskii-Pukhilkov} and present progress and open problems on the third. In Section \ref{sec:piecewise_polynomiality} we show  
that for integral polytopes $P_A(\mathbf{b})$, where $A$ is fixed and $\mathbf{b}\in \Z^n$ varies, the coefficients of $h_{P_A(\mathbf b), w}^{*}$ are piecewise polynomials in terms of $\mathbf b$, and the regions of polynomiality are {\em secondary cones} or {\em type cones} corresponding to different normal fans of polytopes. 
In Section \ref{sec:weighted_counts},
we give an effective algorithm for computing these polynomials for smooth polytopes, and we analyze its computational complexity.  Our code and a catalog of Ehrhart and $h^*$ polynomials for type $A$ alcoved polytopes of small dimension with respect to small degree weights are available on 
our website at\begin{center}\href{https://sites.gatech.edu/weightedehrhart/}{https://sites.gatech.edu/weightedehrhart/}.\end{center}
A similar catalog for alcoved polytopes in the unweighted case was given by Brandenburg, Elia, and Zhang~\cite{Brandenburg_2023}.
In Section \ref{sec:hstar_positivity} we investigate the signs of coefficients of $h_{P, w}^{*}$. We observe that the space of polytopes having nonnegative weighted $h^{*}$ is not convex and that all sign patterns of the coefficients of $h_{P, w}^{*}$ are possible in small cases. We also show that the roots of the weighted $h^*$ polynomial converge to the roots of the Eulerian polynomial as integral polytopes are dilated.

\section{Piecewise Polynomiality}
\label{sec:piecewise_polynomiality}

We first recall some key notions from~\cite{Khovanskii-Pukhilkov-Virtual-Polytopes}.
Let $\mathcal{P}$ be the set of all integral polytopes in $\R^d$.
A \textbf{finitely additive measure} on $\mathcal{P}$ is a mapping $\phi: \mathcal{P} \rightarrow \mathbb{R}$ such that for any 
collection of polytopes $P_{1}, \ldots, P_{k} \in \mathcal{P}$ such that $\bigcup_{i = 1}^{k} P_{i} \in \mathcal{P}$ and $\bigcap_{j=1}^\ell P_{i_j}\in \mathcal{P}$ for any $i_1<\dots<i_\ell$ , we have the following inclusion-exclusion property:
$$\phi\left(\bigcup_{i = 1}^{k} P_{i} \right) = \sum_{i = 1}^{k}\phi(P_{i}) - \sum_{i < j}\phi(P_{i} \cap P_{j}) + \cdots + (-1)^{k-1} \phi(P_1\cap\cdots\cap P_k).$$

A mapping $f : L \rightarrow R$ of abelian groups is called a {\bf polynomial of degree} $\leq m$ if either $m=0$ and $f$ is constant, or $m \geq 1$ and for each $\ell \in L$, the map $f_\ell : L \rightarrow R$ given by $f_\ell(\mathbf{y}) = f(\mathbf{y}+\ell) - f(\mathbf{y})$ is a polynomial of degree $\leq m-1$.
A finitely additive measure $\phi$ is \textbf{polynomial of degree $\leq m$} if for each $P \in \mathcal{P}$, the map $\Z^d \rightarrow \R$ given by $\ell \mapsto \phi(P+\ell)$ is a polynomial of degree $\leq m$.

\begin{lem}
For any polynomial weight $w$ of degree $m$, the weighted integer point counting function
is a finitely additive measure on $\mathcal{P}$, and it is polynomial of degree $\leq m$.  
\end{lem}

\begin{proof}
Being a finitely additive measure follows immediately from the definition. 
To check polynomiality, fix $P\in \mathcal{P}$ and  define $f: \mathbb{Z}^{d} \rightarrow \mathbb{R}$  by $\ell \mapsto \sum_{\mathbf{p} \in (P + \ell) \cap \mathbb{Z}^{d}} w(\mathbf{p})$.  Then for each $\ell\in \Z^d$ and $\mathbf y = (y_1,\dots,y_d)$, we have
\begin{align*}
    f_{\ell}(\mathbf y) :\!&= f(\mathbf y + \ell) - f(\mathbf y) 
    = \sum_{\mathbf{p} \in P \cap \mathbb{Z}^{d}} \left( w(\mathbf{p} + \mathbf y + \ell) - w(\mathbf{p} + \mathbf y) \right)
\end{align*}
which is a polynomial of degree $\leq m - 1$ in $\mathbf y$ variables since $w$ is a polynomial of degree $m$, and the leading terms of $w(\mathbf{p + y} + \ell)$ and $w(\mathbf{p + y})$ with respect to $\bf y$ cancel out.
\end{proof}

A polytope $P'$ is obtained from a polytope $P$ by a \textbf{motion of the walls} if the normal fan of $P'$ is a coarsening of the normal fan of $P$. 
See Figure \ref{fig:alcoved_motion_of_the_walls}. For a polytope $P = P_A(\mathbf{b}^0)$ defined by $A x \leq \mathbf{b}^0$, the set of all $\mathbf{b} \in \Z^n$ such that $P_A(\mathbf{b})$ is obtained from $P$ by a motion of the walls is a polyhedral cone, often known as a {\bf type cone} of $P$.  It is a secondary cone of the vector configuration consisting of the rows of $A$.

\begin{figure}[!ht]
    \centering
    \includegraphics[width=0.5\linewidth]{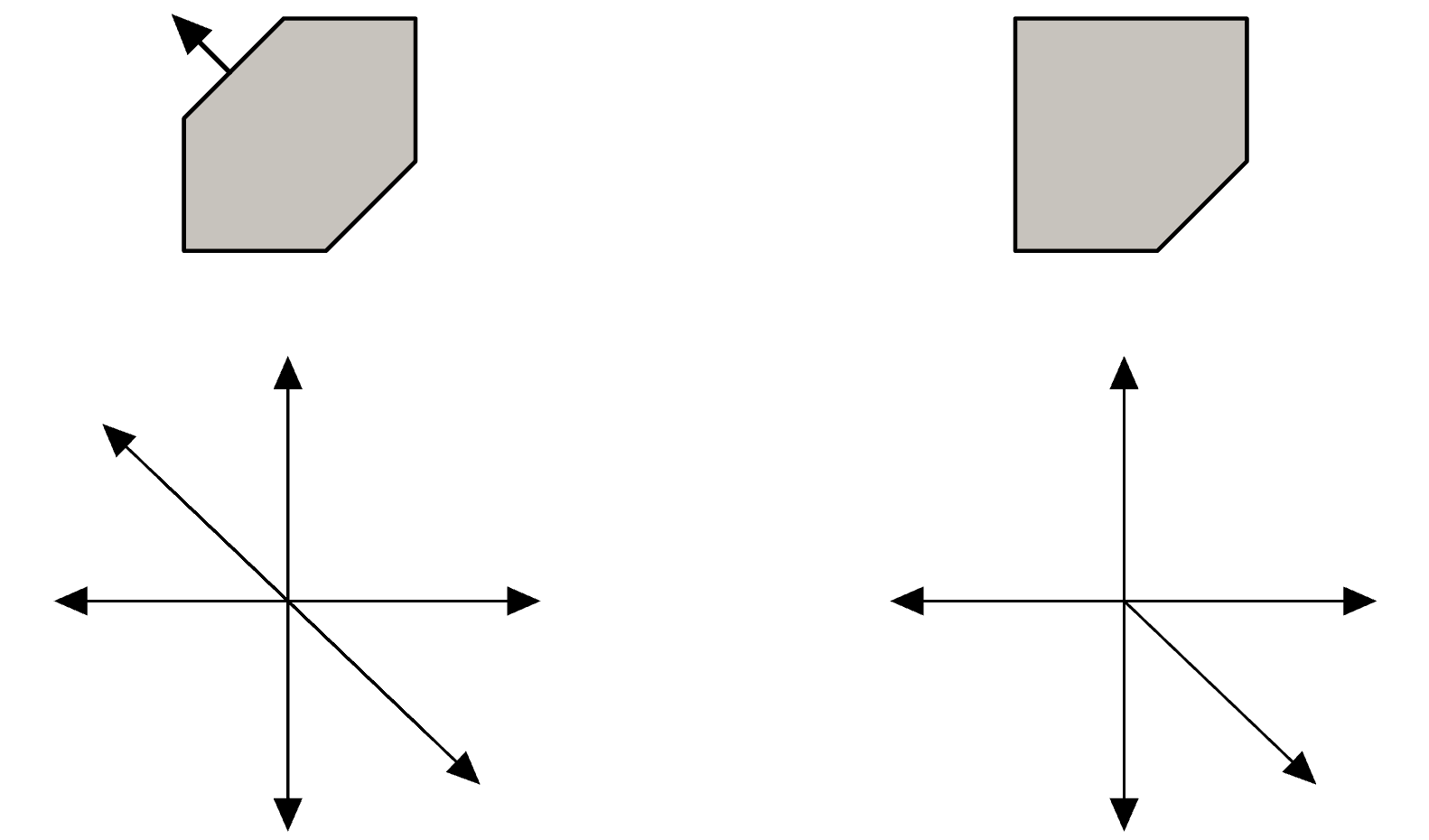}
    \caption{
    An example of $P'$ (top right) obtained by a motion of the walls of $P$ (top left).  Observe that $P'$'s normal fan (bottom right) is a coarsening of that of $P$ (bottom left).
    }
    \label{fig:alcoved_motion_of_the_walls}

\end{figure}


\begin{thm}[Theorem 1, \S 2, \cite{Khovanskii-Pukhilkov-Virtual-Polytopes}]\label{KP 1}
Let
$\phi: \mathcal{P} \rightarrow \mathbb{R}$ be any finitely additive measure which is polynomial of degree $\leq m$.
Fix a matrix $A$ and a polytope $P \in \mathcal{P}$. On the set of integer vectors $\mathbf{b} \in \Z^n$ such that $P_A(\mathbf{b}) \in \mathcal{P}$ is an integral polytope obtained as a motion of the walls from $P$, the map $\mathbf{b} \mapsto \phi(P_A(\mathbf{b}))$ is a real-valued polynomial  of degree $\leq d + m$ in $b_{1}, \ldots, b_{n}$.
\end{thm}

By applying this theorem to an arbitrary integer polytope and our weighted integer point counting function, we immediately obtain the following.
\begin{prop}
Fix an $n\times d$ integer matrix $A$, a polytope $P \subseteq \mathbb{R}^{d}$, and a polynomial weight $w$ of degree $m$.  On the set of integer vectors $\mathbf{b} \in \Z^n$ such that $P_A(\mathbf{b}) \in \mathcal{P}$ is an integral polytope obtained as a motion of the walls from $P$, the weighted integer point counting function $\mathbf{b} \mapsto \sum_{\mathbf{p} \in P_A(\mathbf{b}) \cap \mathbb{Z}^{d}} w(\mathbf{p})$ is a polynomial of degree $\leq d + m$ in $b_1, \dots, b_n$.
\end{prop}

In other words, the weighted integer point counting function is a piecewise polynomial, where the pieces are type cones. 
Furthermore, when $A$ is unimodular, i.e.\ when it has rank $d$ and all $d\times d$ minors have determinant $0$ or $\pm1$, all polytopes $P_A(\mathbf b)$ are integral whenever $\mathbf b$ is an integer vector.   Thus we obtain the following.
\begin{cor}
\label{cor:weighted_ehrhart_polynomial_from_weighted_count}
Fix an $n \times d$ unimodular matrix $A$.  For $\mathbf{b} \in \Z^n$, the coefficients of the weighted Ehrhart and $h^*$-polynomials of the polytopes $P_A(\mathbf{b})$ are piecewise polynomials in $b_{i}$. 
\end{cor}

\section{Algorithms and computations for smooth polytopes}
\label{sec:weighted_counts}

A $d$-dimensional integral polytope $P \subset \R^d$ is called \textbf{smooth} if each vertex cone of $P$ is generated by a basis of $\Z^d$.  For a unimodular matrix $A$, a polytope $P_A(\mathbf b)$ is smooth if and only if it is integral and simple.
We will now recall an {\it Euler--Maclaurin} type formula, due to Khovanskiǐ and Pukhlikov, that relates a weighted sum of lattice points in a smooth polytope to certain differential operators applied to integrals.  Generalizations exist for non-smooth polytopes, but they are significantly more complicated~\cite{BrionVergne, BerlineVergne}.




\begin{defin}
The \textit{Todd operator} is a differential operator defined as $$\Todd_h=\sum_{k\geq 0}(-1)^k \frac{B_k}{k!}\left(\frac{d}{dh}\right)^k,$$  where $B_k$ are the Bernoulli numbers satisfying $\displaystyle\frac{z}{e^z-1}=\sum_{k\geq 0} \frac{B_k}{k!}z^k.$
\end{defin}

With respect to multiple variables $\textbf{h} = (h_1, \ldots, h_{n})$, the \textbf{multivariate Todd operator} is
\(\Todd_{\textbf{h}} = \prod_{i = 1}^{n} \Todd_{h_{i}}.\)
It is linear and preserves polynomiality.

The following is a weighted version of an Euler--Maclaurin type formula due to Khovanskiǐ and Pukhlikov~\cite{Khovanskii-Pukhilkov}, stated as in \cite[Theorem 12.6]{beckrobins}, which allows us to compute the weighted integer point count by applying the multivariate Todd operator to the integral of the weight over the polytope after a perturbation.
For variables ${\mathbf x} = (x_{1}, \ldots, x_{d})$, we denote $e^{z_{1} x_{1} + \ldots + z_{d} x_{d}}$ as $\text{exp}(\mathbf{x \cdot z})$.

\begin{thm}[Weighted Khovanskiǐ–Pukhlikov]
\label{thm:weighted-khovanskii-pukhilkov}
For any smooth $d$-polytope of the form $P = P_A(\mathbf{b})$ and homogeneous polynomial weight $w,$
$$\sum_{\mathbf{p}\in P \cap \Z^d} w(\mathbf{p})=\Bigg(\operatorname{Todd}_\mathbf{h} \int_{P_A(\mathbf{b+h})} w(\mathbf{x}) d\mathbf{x}\Bigg)\Bigg|_{\mathbf{h}=0}.$$
\end{thm}
\begin{proof} It suffices to prove this when $w$ is a monomial.
We will use the formulation in \cite[Theorem 12.6]{beckrobins} of the Khovanskiǐ–Pukhlikov formula:
\begin{equation}\label{KP}
    \sum_{\mathbf{p}\in P} \exp(\mathbf{z}\cdot \mathbf{p})=\Bigg(\operatorname{Todd}_\mathbf{h} \int_{P_A(\mathbf{b}+\mathbf{h})} \exp(\mathbf{x}\cdot \mathbf{z}) d\mathbf{x} \Bigg)\Bigg|_{\mathbf{h}=0}
\end{equation} for any smooth polytope $P.$
For $w=x_1^{a_1}\cdots x_d^{a_d}$, define $\partial_{w}=\frac{\partial^{a_1}}{\partial z_1^{a_1}}\circ \cdots \circ \frac{\partial^{a_d}}{\partial z_d^{a_d}}$.    Then 
\begin{align*}
\label{eqn:weighted_count_as_partial}
( \partial_{w} \sum_{\mathbf{p}\in P} \exp(\mathbf{z} \cdot \mathbf{p})  )\Big|_{\mathbf{z}=0} 
=  \sum_{\mathbf{p}\in P}  w(\mathbf{p}).
\end{align*} 
The desired result is obtained by applying $\partial_w$ to the right-hand side of Equation (\ref{KP}) and setting $\mathbf z = 0$. The three operations $\partial_w,$ $\Todd_h$, and $\int dx$ commute, since they are with respect to different variables.  
\end{proof}


This allows us to explicitly compute the formulas whose existence was proven in Corollary \ref{cor:weighted_ehrhart_polynomial_from_weighted_count}.  To do so, we discuss how to compute the integral $\int_{P_A(\mathbf{b+h})} w(\mathbf{x}) d\mathbf{x}$ as a polynomial in $\mathbf{b+h}$ in a type cone.

Fix a polytope of the form $P=P_A(\mathbf b^0)$ and a triangulation $T$ of $P$.  As $P$ deforms to $P_A(\mathbf b + \mathbf h)$ by a motion of the walls, its vertices move as polynomials in $\mathbf{b} + \mathbf{h},$ and $T$ gets deformed accordingly as well. The integral $\int_{P_A(\mathbf{b} + \mathbf{h})} w(x) dx=\sum_{\Delta\in T(\mathbf{b} + \mathbf{h})} \int_\Delta w(x) dx$ is then computable as a polynomial in $\mathbf{b}$, using the following.

\begin{thm}[Corollary 20, \cite{how-to-integrate-a-polynomial-over-a-simplex}]
\label{thm:how-to-integrate-a-polynomial-over-a-simplex}
Let $f$ be a homogeneous polynomial of degree $m$ in $d$ variables, and let $\mathbf{s}_{1}, \ldots, \mathbf{s}_{d+1}$ be the vertices of a $d$-dimensional simplex $\Delta$. Then
$$\int_{\Delta} f(\mathbf{y}) d\mathbf{y} = \frac{\operatorname{vol}(\Delta)}{2^m m! \binom{m+d}{m}} \sum_{1 \leq i_{1} \leq i_{2} \leq \ldots \leq i_{m}\leq d+1} \, \sum_{\epsilon \in \{\pm 1\}^{m}} \epsilon_{1}\epsilon_{2} \ldots \epsilon_{m}f(\sum_{k = 1}^{m} \epsilon_{k} \mathbf{s}_{i_{k}}).$$
\end{thm}

Putting them together, we obtain Algorithm~\ref{alg:cap} for computing the parametric weighted integer point count for a family of smooth polytopes in a type cone. 
The weighted Ehrhart polynomial of $P_A(\mathbf{b})$ is obtained by evaluating the weighted count at $b_{i} = tb_{i}$.  The algorithm has been implemented in SageMath notebook, availabe on our \href{https://sites.gatech.edu/weightedehrhart/}{website}.

\begin{algorithm}
\begin{algorithmic}
\caption{Parametric Weighted Integer Point Count}\label{alg:cap}
\Require A homogeneous polynomial weight  $w$ in $d$ variables, a $n \times d$ integer matrix $A$, and a vector $\mathbf{b}^0 \in \Z^n$ such that  $P_A(\mathbf{b}^0)$ is a smooth polytope.
\Ensure The polynomial in $\mathbf{b}$ giving the weighted integer point count $\sum_{\mathbf{p} \in P_A(\mathbf{b}) \cap \mathbb{Z}^{d}} w(\mathbf{p})$ for all integral polytopes $P_A(\mathbf{b})$ obtained from $P_A(\mathbf{b}^0)$ by a motion of the walls.  

\State {\bf Step 1:} Find a triangulation of $P_A(\mathbf{b}^0)$, which gives a triangulation of $P_A(\mathbf{b+h})$ by a motion of the walls from $P_A(\mathbf b^0)$.  Write its vertices as polynomials in $\mathbf{b}$ and $\mathbf{h}$.
\State {\bf Step 2:} For each maximal simplex $\Delta$ in the triangulation of $P_A(\mathbf{b+h})$, compute the integral $\int_{\Delta} w(\mathbf{x}) d\mathbf{x}$ using Theorem~\ref{thm:how-to-integrate-a-polynomial-over-a-simplex}.  Sum up integrals over all maximal simplices to obtain $\int_{P_A(\mathbf{b}+\mathbf{h})} w(\mathbf{x}) d\mathbf{x}$ as   a polynomial in $\mathbf{b}$ and $\mathbf{h}$.
\State {\bf Step 3:} Apply the Todd operator to this integral and evaluate at $\mathbf{h} = 0$, which gives the parametric weighted count as a polynomial in $\mathbf b$, by Theorem \ref{thm:weighted-khovanskii-pukhilkov}.
\end{algorithmic}
\end{algorithm}

\subsection{Example: alcoved polytopes}
\label{sec:example}


In this section, we specialize our algorithm to (type $A$) alcoved polytopes which are also known as {\em polytropes}.
A $d$-dimensional \textbf{alcoved polytope} is an integral polytope in $\R^d$ defined by inequalities of the form \begin{equation}
\label{eqn:alcovedineqs}
    x_i-x_j\leq b_{ij}
\end{equation} where $b_{ij} \in \Z$ for all pairs of distinct $i,j\in[d+1]$.
We set $x_{d+1}=0$ by convention. 
Every alcoved polytope is associated with a unique set of parameters $\mathbf b = (b_{ij})$ such that $b_{ij}+b_{jk}\geq b_{ik}$ for all distinct $i,j,k\in [d+1].$ Equivalently, these conditions mean that all the inequalities \eqref{eqn:alcovedineqs} are tight. The set of all tuples satisfying these inequalities forms a cone (called the {\em metric cone} or the {\em alcoved cone}) in $\R^{d(d+1)}$ and has a fan structure $\F_d$ given by the combinatorial types of alcoved polytopes  \cite{tran2016enumeratingpolytropes,early2025}. 
Alcoved polytopes in the interior of a maximal cone of $\F_d$ are called \textbf{maximal}. They are smooth with $\binom{2d}{d}$ vertices. We provide a function in our code that takes in a non-maximal alcoved polytope $P_A(\mathbf{b})$ and returns a maximal alcoved polytope $P_{A}(\mathbf{b}')$ whose normal fan refines that of $P_{A}(\mathbf{b})$.

\begin{figure}[h]
    \centering
    \includegraphics[width=0.5\linewidth]{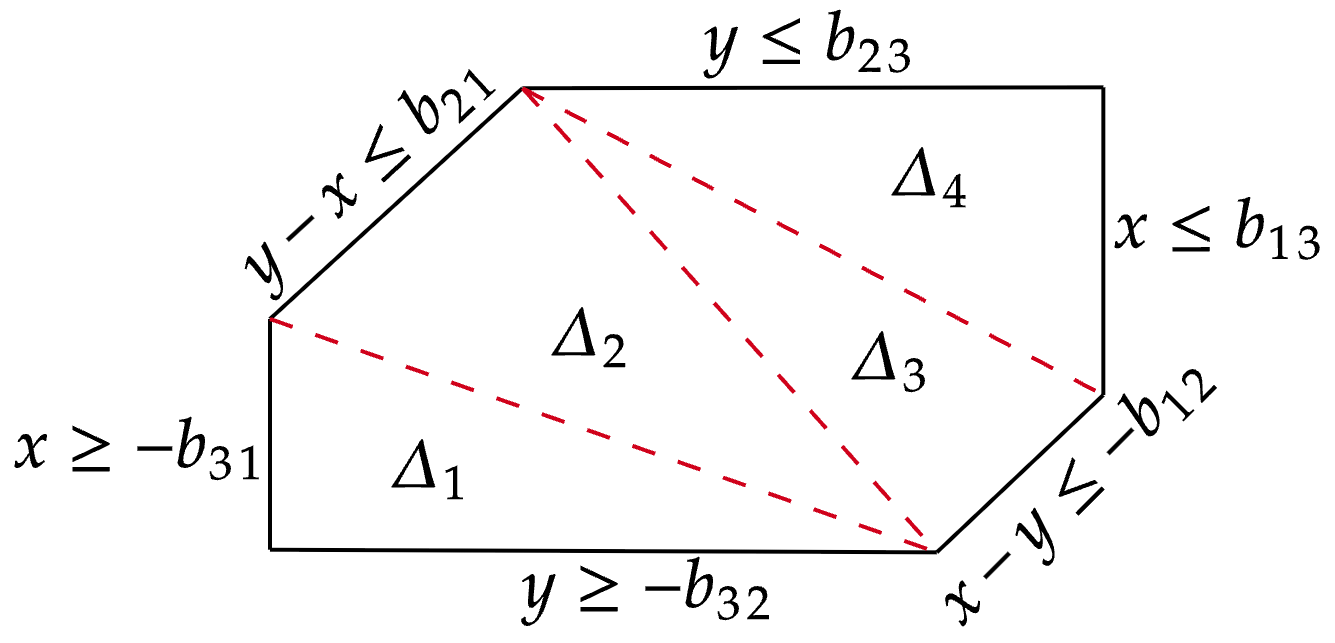}
    \caption{A fixed triangulation of a 2-dimensional alcoved polytope}
    \label{2D triangulation}
\end{figure}

We now demonstrate our algorithm by computing this parametric formula for 2-dimensional alcoved polytopes $P_A(\mathbf{b})$ with weight $w(\mathbf{p})=p_1p_2$.  We compute the integral $\int_{P_A(\mathbf{b} + \mathbf{h})}w(\mathbf{x})d\mathbf{x}$ by fixing a triangulation of the polytope (Figure \ref{2D triangulation}), integrating over each simplex
using Theorem \ref{thm:how-to-integrate-a-polynomial-over-a-simplex} with $d=m=2$, and  
summing over each simplex, to obtain the integral as a polynomial formula in $\mathbf b$ and $\mathbf h$.  Then we iteratively apply the Todd operator 6 times, once for each $h_{ij},$ and set each $h_{ij}=0,$  to obtain
{\footnotesize\begin{align*} 
\sum_{\mathbf{p}\in P_A(\mathbf{b})}w(\mathbf{p}) &=  
\frac{1}{24}
\big( b_{12}^4 - 6b_{12}^2b_{13}^2 + 8b_{12}b_{13}^3 - 3b_{13}^4 + b_{21}^4 + 6b_{13}^2b_{23}^2 - 6b_{21}^2b_{23}^2 + 8b_{21}b_{23}^3 - 3b_{23}^4 - 6b_{21}^2b_{31}^2 \\
& + 8b_{21}b_{31}^3 - 3b_{31}^4 - 6b_{12}^2b_{32}^2 + 6b_{31}^2b_{32}^2 + 8b_{12}b_{32}^3 - 3b_{32}^4 + 2b_{12}^3 - 6b_{12}^2b_{13} + 6b_{12}b_{13}^2 - 2b_{13}^{3} + 2b_{21}^{3}\\ & 
+ 6b_{13}^2b_{23} - 6b_{21}^2b_{23} + 6b_{13}b_{23}^2 + 6b_{21}b_{23}^2 - 2b_{23}^3 - 6b_{21}^2b_{31} + 6b_{21}b_{31}^2 - 2b_{31}^{3} - 6b_{12}^{2}b_{32} + 6b_{31}^{2}b_{32} \\
& + 6b_{12}b_{32}^2 + 6b_{31}b_{32}^2 - 2b_{32}^3 - b_{12}^2 - 2b_{12}b_{13} + 3b_{13}^2 - b_{21}^2 + 6b_{13}b_{23} - 2b_{21}b_{23} + 3b_{23}^2 - 2b_{21}b_{31} \\
&  + 3b_{31}^2 - 2b_{12}b_{32} + 6b_{31}b_{32} + 3b_{32}^2 - 2b_{12}  + 2b_{13}  - 2b_{21} + 2b_{23} + 2b_{31} + 2b_{32}\big)
\end{align*}}

By replacing each $b_{ij}$ with $tb_{ij},$ we obtain a formula for the parametric weighted Ehrhart polynomial $\ehr_{P_A(\mathbf{b}), w}(t),$ which can then be used to find the  parametric weighted $h^*$-polynomial $h^*_{P_A(\mathbf{b}), w}(z).$  

More details of this example, as well as a catalog of all parametric weighted Ehrhart and $h^*$-polynomials for 2-dimensional and 3-dimensional alcoved polytopes with all monomial weights up to degree 5 and degree 3, respectively, are available on our \href{https://sites.gatech.edu/weightedehrhart/}{website}.

\subsection{Computational Complexity}
\label{sec:computational_complexity}

Any algorithm for computing parametric weighted Ehrhart polynomials is not expected to be polynomial with respect to its dimension and the degree of the weight, since the number of possible terms in its output grows exponentially.  However, for a fixed dimension $d$ and degree $m$, our algorithm has polynomial running time in the number of facets $n$ of $P_{A}(\mathbf{b})$. 

\begin{thm} Fix a dimension $d.$
For a smooth polytope $P=P_A(\mathbf{b}^0)\in \R^d$ with $n$ facets and a monomial weight $w$ of degree $m,$ our algorithm to compute $\sum\limits_{\mathbf{p}\in P_A(\mathbf b)\cap \Z^d}w(\mathbf{p})$  runs in $O\left( n^{d^2/4}\cdot 2^m\cdot  \binom{d+m}{d} \cdot (d+m)\right)$ time. 
\end{thm}

\begin{proof}
First, our algorithm triangulates $P.$  If $|V(P)|$ is the number of vertices in $P,$ then $P$ can be triangulated in $O(|V(P)|^{\lfloor d/2\rfloor})$ time \cite{Dewall_1998}, while $|V(P)|=O(n^{\lfloor d/2\rfloor})$ by McMullen's Upper Bound Theorem \cite{McMullen_upper_bound}, so triangulation takes $O(n^{d^2/4})$ time. The number of simplices in this triangulation is also  bounded by $O(n^{d^2/4}).$  For each simplex $\Delta$ we compute the integral $\int_\Delta w(x)dx$. By Theorem \ref{thm:how-to-integrate-a-polynomial-over-a-simplex}, each integral takes $O\left(2^m\binom{d+m}{d}\right)$ time. Finally, we add all these integral polynomials together, substitute $\mathbf{b}\mapsto\mathbf{b}+\mathbf{h}$, apply the Todd operator, and simplify the expression. 
By Theorem \ref{thm:how-to-integrate-a-polynomial-over-a-simplex}, the number of terms we store before simplifying the added terms is $O(n^{d^2/4} \cdot  2^m \cdot \binom{d+m}{d}).$  
Substitution and simplification are linear in the number of terms, while the Todd operator takes $O(d+m)$ time per term, since it is linear in the degree of the input. 

\end{proof}

Our algorithm significantly outperforms the method of polynomial interpolation under mild assumptions. To interpolate a degree $d+m$ polynomial in $n$ variables, we build an $N\times N$ matrix $M,$ where $N=\binom{n+d+m}{n}$ and  $M_{ij}$ is the evaluation of the $j^{\text{th}}$ monomial of degree $d+m$ in $n$ variables at the right hand side $\mathbf b^{(i)}$ for a randomly generated polytope $P_i=P_A(\mathbf{b}^{(i)})$ of the desired type.  Then we solve $M \mathbf{x}=\mathbf{c}$, where $c_i=\sum\limits_{\mathbf{p}\in P_i\cap \Z^d}w(\mathbf{p})$ and $\mathbf{x}$ gives the coefficients of our parametric formula.  

Each $c_i$ can be computed in $\text{poly}(n,m)$ time by a weighted version of Barvinok's algorithm~\cite{Barvinok_Pommersheim},  
giving interpolation a running time of  $T_2=\Omega(N^2+N\cdot \text{poly}(n,m)).$   Since $d$ is fixed, the running time of Algorithm \ref{alg:cap} is $T_1=O(2^m\cdot\text{poly}(n,m))$.  If $m=o(n)$ (which is often the case in practice), then  $N=\binom{n+d+m}{d+m}\sim \left(\frac{en}{m}\right)^m\gg 2^m,$ so $T_2\gg T_1$  
 and Algorithm \ref{alg:cap} is exponentially faster than interpolation.  It is not clear how to {\em efficiently} adapt Barvinok's algorithm to the parametric and weighted setting.

\section{Signs of coefficients of weighted $h^{*}$-polynomials}
\label{sec:hstar_positivity}

\subsection{Weighted $h^*$ positivity regions}

Which combinations of integral polytopes $P=P_A(\mathbf b)$ and homogeneous polynomial weights $w$ give weighted
$h^*$-polynomials with nonnegative coefficients? Alternatively, when does $h_{P, w}^*(z) \geq 0$ when $z \geq 0$? These are difficult problems for which we do not expect a simple answer.  In~\cite{Bajo_2024}  some sufficient conditions for positivity are given for a fixed polytope.  
Since the coefficients of weighted $h^*$-polynomials are piecewise polynomials, the $h^*$ positivity regions are described by polynomial inequalities in $\mathbf b$ and the coefficients of the weight polynomials.

\begin{lem}
For a fixed integral polytope $P$, the set of all weights giving rise to a weighted $h^*$-polynomial with nonnegative coefficients forms a convex polyhedral cone. 
\end{lem}

\begin{proof}
    For a fixed $P$, the map sending a weight polynomial $w$ to its corresponding weighted $h^*$-polynomial $h^*_{P,w}$ is a linear transformation.  The set of weights giving rise to nonnegative $h^*$-polynomials form the  preimage of the nonnegative orthant under this linear map, so it is a convex polyhedral cone.
\end{proof}

On the other hand, we can fix the weight polynomial and vary the polytopes in some parameter space.  Is the set of integral polytopes with positive weighted $h^*$-polynomial {\em convex}?  There are at least two natural ways to take convex combination of polytopes:  For lattice polytopes $P$ and $P'$, we can take the scaled Minkowski sums $\lambda P + (1-\lambda)P'$ for $0 \leq \lambda \leq 1$.  Alternatively, for two polytopes $P_A(\mathbf{b})$ and $P_A(\mathbf{b}')$, we can consider polytopes of the form $P(\lambda \mathbf{b} + (1-\lambda) \mathbf{b}')$ for $0 \leq \lambda \leq 1$. These two notions coincide when the two polytopes are in the same type cone, but they differ in general.  For alcoved polytopes, the latter notion is called the {\em alcoved hull} of the former. They coincide for $2$-dimensional alcoved polytopes, since there is only one maximal type cone. See~\cite{early2025} for more discussion.

\begin{ex}
\label{ex:nonconvex}
This example shows that for a fixed polynomial weight, the set of lattice polytopes with nonnegative weighted $h^*$-polynomials is not convex, even with a linear weight.

Let $w=-3x+2y.$  Let $A = \begin{pmatrix} 
1&1&-1&0&-1&0 \\ -1&0&1&1&0&-1\end{pmatrix}^T$
arising from the facets of two-dimensional alcoved polytopes.  Let $\mathbf{v}=(3,5,4,8,3,0),$ $\mathbf{h}=(-1,2,0,1,0,0),$ and for $i=0,1,2,3$, let $P_i=P_A(\mathbf{v}+i\mathbf{h})$.   See Figure~\ref{fig:nonconvex}. The integral polygons $P_1$ and $P_2$ are Minkowski convex combinations of $P_0$ and $P_3$, and equivalently correspond to convex combinations of the right hand side $\mathbf{b}$ vectors.  We computed that:
\begin{align*}
    h^*_{w,P_0}&= 10z^3 + 65z^2 + 25z &
  h^*_{w,P_1} & =-10z^3 - 37z^2 - 7z \\
    h^*_{w,P_2} &=-9z^3 - 39z^2 - 10z &
    h^*_{w,P_3} &=15z^3 + 67z^2 + 18z
\end{align*}
 The polytopes $P_0$ and $P_3$ have positive weighted $h^*$ coefficients, while $P_1$ and $P_2$ do not. \qed
\begin{figure}[h]
    \centering
    \includegraphics[width=0.4\linewidth]   {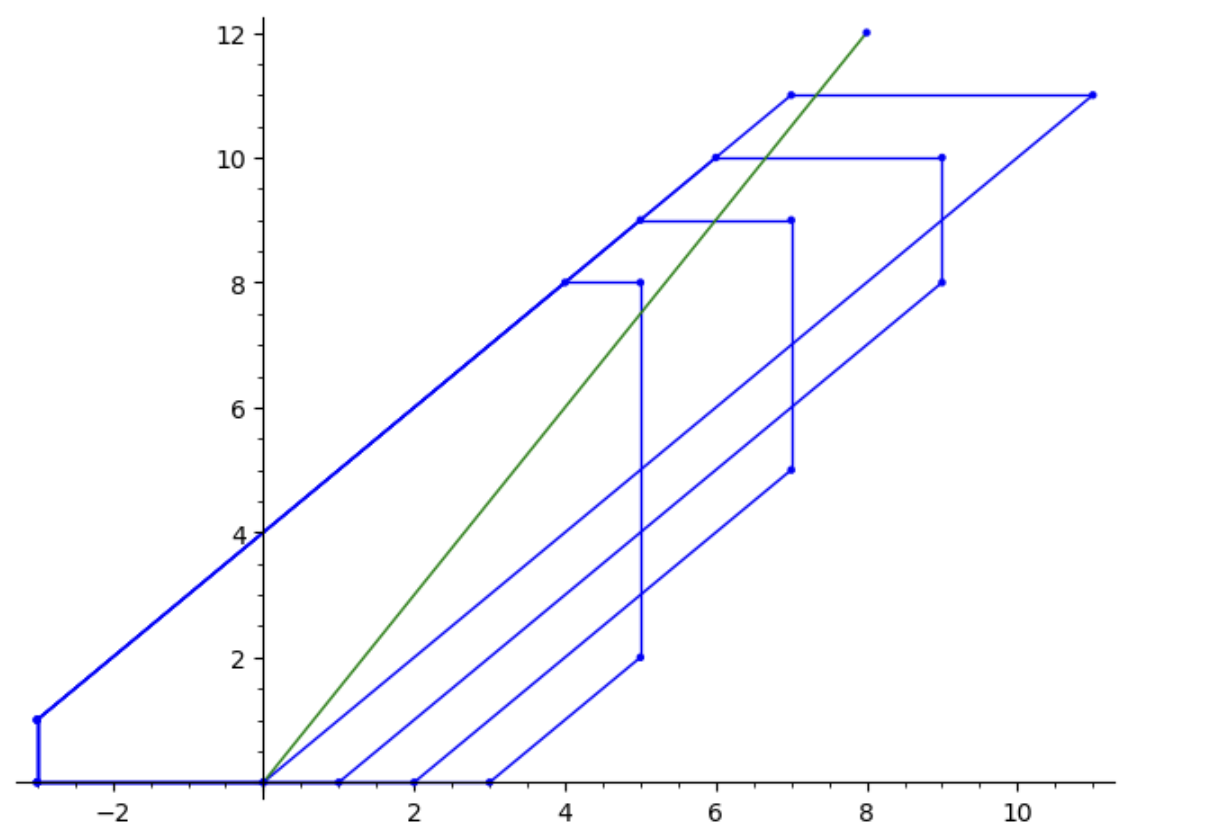} \caption{From left to right, the four blue polygons show $P_0,P_1,P_2,P_3$ from Example~\ref{ex:nonconvex}. The zero locus of the linear weight $w=2y-3x$ is the green line. 
    }
    \label{fig:nonconvex}
\end{figure}
\end{ex}

\subsection{Sign patterns and linear spaces of $h^{*}$- coefficients}

For a fixed full-dimensional lattice polytope $P \subset \R^d$ and a fixed integer $m \geq 0$, the map that sends a degree $m$ polynomial $w$ to the polynomials $h^*_{P,w}$ is a linear map from the vector space of degree-$m$ polynomials in $d$ variables to the vector space of degree $d+m$ polynomials in one variable. Let us denote the image by $L_{P,m}$.  If $m > 0$, the $h^*$-polynomials have constant term $0$. When $L_{P,m}$ is all of the $(d+m)$-dimensional space of univariate polynomials with constant term $0$, then trivially all sign patterns are possible for the weighted $h^*$ polynomial.
The degree-$m$ polynomials in $d$ variables form a vector space of dimension $\bigl(\!\binom{d}{m}\!\bigr) = \binom{d+m-1}{m}$, so for a generic polytope $P$, we expect
\[
\dim L_{P,m} = \min\left( \binom{d+m-1}{m}, d+m \right)
\]
The only cases when the expected dimension is less than $d+m$ are when $d=2$ or $m=1$.  
We verified computationally that every sign pattern is possible for weighted $h^*$-polynomials of alcoved polytopes for dimension 2, degrees 1 to 5, and dimension 3, degree 1; see our \href{https://sites.gatech.edu/weightedehrhart/}{website}.  Some sign patterns are more ubiquitous than others. The alternating sign pattern is particularly rare.
\begin{que}
   For a fixed dimension $d$ and degree $m$, is every sign pattern possible for weighted $h^*$-vectors as we vary the (alcoved) polytopes and the weights?  
\end{que}

\subsection{Weighted $h^*$-polynomials of large dilates of polytopes}
Let $f(t)$ be a degree $k$ polynomial over $\C$.  Then we have $\sum_{t\geq 0} f(t) z^t = \frac{h(z)}{(1-z)^{k+1}}$ for some polynomial $h(z)$ of degree $\leq k$.  For a positive integer $r$, let us define a polynomial $h^{<r>}(z)$ by
\[
\sum_{t \geq 0}f(rt)z^t = \frac{h^{<r>}(z)}{(1-z)^{d+1}}
\]
If $f(t)$ is the Ehrhart polynomial of a polytope $P$, then $h^{<r>}(z)$ is the $h^*$-polynomial of the dilation $rP$.
Brenti and Welker described the linear map $h(z) \mapsto h^{<r>}(z)$ and showed that the roots of $h^{<r>}(z)$ converges to the roots of the Eulerian polynomial~\cite{BrentiWelker}.  Let us give a derivation. Let $f(t) = \sum_{i=0}^k c_i t^i$ be a degree $k$ polynomial.  Then
\[
\sum_{t \geq 0}f(rt)z^t = \sum_{t \geq 0}  \sum_{i=0}^k c_i r^i t^i z^t = \sum_{i=0}^k c_i r^i \sum_{t\geq 0}t^i z^t = \sum_{i=0}^k c_i r^i \frac{A_i(z)}{(1-z)^{i+1}}
\]
where $A_i$ are Eulerian polynomials.  Thus
\[
h^{<r>}(z) = c_k r^k A_k(z) + c_{k-1} r^{k-1} (1-z)A_{k-1}(z) + \cdots + c_0 (1-z)^k.
\]
Let $\widetilde{h^{<r>}}(z)$ be the monic polynomial which is a constant multiple of $h^{<r>}(z)$, that is, $\widetilde{h^{<r>}}(z)=h^{<r>}(z)/(c_kr^k - c_{k-1}r^{k-1} + \cdots \pm c_0)$.  Clearly it has the same roots as $h^{<r>}(z)$.  If $c_k \neq 0$, then as $r\rightarrow \infty$, the polynomial $\widetilde{h^{<r>}}(z)$ converges to $A_k(z)$, so
the roots of $h^{<r>}(z)$ converge to the roots of $A_k(z)$.  Let us now apply this to weighted $h^*$-polynomials.

\begin{thm}
\label{thm:dilates}
Let $P$ be an integral $d$-polytope and $w$ be a homogeneous degree $m$ polynomial such that $h^*_{w,P}(1) \neq 0$. As integers $r\to \infty$ the roots of $h^*_{rP,w}(z)$ limit to the roots of the Eulerian polynomial $A_{d+m}(z)$.  It follows that for sufficiently large $r$, the coefficients of $h^*_{rP,w}(z)$ all have the same sign, with unimodal and log-concave absolute values.
\end{thm}

\begin{proof} Since $h^*_{w,P}(1) \neq 0$, the weighted Ehrhart polynomial has degree $d+m$.
As explained in the paragraph above the theorem, the roots of $h^*_{rP,w}(z)$ converge to the roots of $A_{d+m}(z)$.
Since the complex roots of $h^*_{rP,w}(t)$ come in conjugate pairs, and the Eulerian polynomial is real rooted with negative simple roots, we cannot have pairs of non-real roots converging to the same real root at the limit. Thus for sufficiently large $r$ the polynomial $h^*_{w,rP}$ is also real rooted with simple negative real roots.  In particular, its coefficients are either all-positive or all-negative, depending on the sign of the leading coefficient of the weighted Ehrhart polynomial.  Consequently, the absolute values of the coefficients are unimodal and log-concave for all sufficiently large integers $r$.
\end{proof}

\begin{que}
Fix a homogeneous weight polynomial $w$ and dimension $d$. Is there a universal bound $N_{w,d}$ such that for every lattice polytope $P$ in $\R^d$, the polynomial $h^*_{rP,w}$  has simple negative real roots for all $r\geq N_{w,d}$?  or has coefficients which have the same sign? or are unimodal? or log-concave?
\end{que}
For unweighted Ehrhart $h^*$-polynomials, affirmative answers have been given in~\cite{BeckStapledon}, where it is conjectured that $d$ dilates suffice for the $h^*$ polynomial to have simple negative roots.
In the weighted case, there exist examples where a large dilate is required.  See the \href{https://sites.gatech.edu/weightedehrhart/}{website} for an example of a $2$-dimensional alcoved polytope and a  degree 1 weight for which at least 63 dilates are needed for the weighted $h^*$ to have coefficients of the same sign and to have simple negative roots.  


\section{Acknowledgements}

The first two authors acknowledge support from the U.S. Department of Education Graduate Assistance in Areas of National Need (GAANN) program at the Georgia Institute of Technology (Award \#P200A240169).  The last author acknowledges support from National Science Foundation (Award \#2348701). We thank Ben Braun for pointing us to the phenomenon in Theorem \ref{thm:dilates} and the reference~\cite{BeckStapledon}.

\bibliographystyle{amsalpha}
\bibliography{refs}

@article{TRAN2016enumeratingpolytropes,
title = {Enumerating polytropes},
journal = {Journal of Combinatorial Theory, Series A},
volume = {151},
pages = {1-22},
year = {2017},
doi = {10.1016/j.jcta.2017.03.011},
author = {Ngoc Mai Tran}
}

@article{Brandenburg_2023,
   title={Multivariate volume, Ehrhart, and h⁎-polynomials of polytropes},
   volume={114},
   DOI={10.1016/j.jsc.2022.04.011},
   journal={Journal of Symbolic Computation},
   publisher={Elsevier BV},
   author={Brandenburg, Marie-Charlotte and Elia, Sophia and Zhang, Leon},
   year={2023},
   month=jan, pages={209–230} }

@article{STURMFELS1995302,
title = {On vector partition functions},
journal = {Journal of Combinatorial Theory, Series A},
volume = {72},
number = {2},
pages = {302-309},
year = {1995},
doi = {10.1016/0097-3165(95)90067-5},
author = {Bernd Sturmfels}
}

@article{Bajo_2024,
    AUTHOR = {Bajo, Esme and Davis, Robert and De Loera, Jes\'us A. and
              Garber, Alexey and Garz\'on Mora, Sof\'ia and Jochemko,
              Katharina and Yu, Josephine},
     TITLE = {Weighted {E}hrhart theory: extending {S}tanley's nonnegativity
              theorem},
   JOURNAL = {Adv. Math.},
  FJOURNAL = {Advances in Mathematics},
    VOLUME = {444},
      YEAR = {2024},
     PAGES = {Paper No. 109627, 30},
   MRCLASS = {52B20 (05A15 52B45)},
  MRNUMBER = {4728517},
MRREVIEWER = {P.\ McMullen},
       DOI = {10.1016/j.aim.2024.109627}
}

@article{ehrhart67,
  author  = "Ehrhart, E.",
  title   = "Sur une problème de géométrie diophantine linéaire.",
  journal = "Journal für die reine und angewandte Mathematik",
  year    = 1967,
  volume  = "1967",
  number  = "227",
  pages   = "25--49",
  DOI = {10.1515/crll.1967.227.25}
}

@incollection{STANLEY1980333,
title = {Decompositions of Rational Convex Polytopes},
series = {Annals of Discrete Mathematics},
publisher = {Elsevier},
volume = {6},
pages = {333-342},
year = {1980},
booktitle = {Combinatorial Mathematics, Optimal Designs and Their Applications},
doi = {https://doi.org/10.1016/S0167-5060(08)70717-9},
author = {Richard P. Stanley},

}

@article{Khovanskii-Pukhilkov,
    AUTHOR = {Pukhlikov, A. V. and Khovanski\u{i}, A. G.},
     TITLE = {The {R}iemann-{R}och theorem for integrals and sums of
              quasipolynomials on virtual polytopes},
   JOURNAL = {Algebra i Analiz},
  FJOURNAL = {Rossi\u iskaya Akademiya Nauk. Algebra i Analiz},
    VOLUME = {4},
      YEAR = {1992},
    NUMBER = {4},
     PAGES = {188--216},
   MRCLASS = {14M25 (14C40 52B20)},
  MRNUMBER = {1190788},
MRREVIEWER = {I.\ Dolgachev},
URL = {https://www.math.toronto.edu/askold/1992-Alg-An-4-english.pdf}
}

@article{Khovanskii-Pukhilkov-Virtual-Polytopes,
    AUTHOR = {Pukhlikov, A. V. and Khovanski\u{i}, A. G.},
     TITLE = {Finitely additive measures of virtual polyhedra},
   JOURNAL = {Algebra i Analiz},
  FJOURNAL = {Rossi\u iskaya Akademiya Nauk. Algebra i Analiz},
    VOLUME = {4},
      YEAR = {1992},
    NUMBER = {2},
     PAGES = {161--185},
   MRCLASS = {52B45 (52B11)},
  MRNUMBER = {1182399},
MRREVIEWER = {E.\ Hertel},
URL = {https://www.math.toronto.edu/askold/1992-Alg-An-2-english.pdf}
}

@book{beckrobins,
    AUTHOR = {Beck, Matthias and Robins, Sinai},
     TITLE = {Computing the continuous discretely},
    SERIES = {Undergraduate Texts in Mathematics},
   EDITION = {Second},
      NOTE = {Integer-point enumeration in polyhedra,
              With illustrations by David Austin},
 PUBLISHER = {Springer, New York},
      YEAR = {2015},
     PAGES = {xx+285},
   MRCLASS = {11P21 (05A15 05B15 11-02 11H06 52B05 52B20)},
  MRNUMBER = {3410115},
       DOI = {10.1007/978-1-4939-2969-6},
}

@misc{early2025,
  author={Early, Nick and Kühne, Lukas and Monin, Leonid},
  eprint={2501.17249},
  eprinttype = {arxiv},
  title={When alcoved polytopes add},
  year={2025},
}

@article{Barvinok_Pommersheim,
title = {An algorithmic theory of lattice points in polyhedra},
journal = {New Perspectives in Algebraic Combinatorics},
volume = {38},
pages = {91-147},
year = {1999},
doi = {10.1017/9781009701587.004},
author = {Barvinok, A.I. and Pommersheim, J.},
}

@article{Dewall_1998,
title = {DeWall: A fast divide and conquer Delaunay triangulation algorithm in Ed},
journal = {Computer-Aided Design},
volume = {30},
number = {5},
pages = {333-341},
year = {1998},
doi = {https://doi.org/10.1016/S0010-4485(97)00082-1},
author = {P Cignoni and C Montani and R Scopigno}
}

@article{how-to-integrate-a-polynomial-over-a-simplex,
    AUTHOR = {Baldoni, Velleda and Berline, Nicole and De Loera, Jesus A.
              and K\"oppe, Matthias and Vergne, Michele},
     TITLE = {How to integrate a polynomial over a simplex},
   JOURNAL = {Math. Comp.},
  FJOURNAL = {Mathematics of Computation},
    VOLUME = {80},
      YEAR = {2011},
    NUMBER = {273},
     PAGES = {297--325},
   MRCLASS = {68W30 (52B55 65Y20)},
  MRNUMBER = {2728981},
       DOI = {10.1090/S0025-5718-2010-02378-6},
}

@article {DeLoeraSturmfels,
    AUTHOR = {De Loera, Jes\'us A. and Sturmfels, Bernd},
    TITLE = {Algebraic unimodular counting},
     BOOKTITLE = {Algebraic unimodular counting},
    PUBLISHER = {Mathematical Programming Society},
    SERIES = {B},
      NOTE = {Algebraic and geometric methods in discrete optimization},
   JOURNAL = {Math. Program.},
  FJOURNAL = {Mathematical Programming},
    VOLUME = {96},
      YEAR = {2003},
    NUMBER = {2},
     PAGES = {183--203},
   MRCLASS = {90C10 (90C27 90C57)},
  MRNUMBER = {1993046},
MRREVIEWER = {Ethan\ D.\ Bolker},
       DOI = {10.1007/s10107-003-0383-9},
}

@article {BrionVergne,
    AUTHOR = {Brion, Michel and Vergne, Mich\`ele},
     TITLE = {Lattice points in simple polytopes},
   JOURNAL = {J. Amer. Math. Soc.},
  FJOURNAL = {Journal of the American Mathematical Society},
    VOLUME = {10},
      YEAR = {1997},
    NUMBER = {2},
     PAGES = {371--392},
   MRCLASS = {11P21 (11H06 14M25)},
  MRNUMBER = {1415319},
MRREVIEWER = {Hartmut\ Menzer},
       DOI = {10.1090/S0894-0347-97-00229-4},
}

@article {BerlineVergne,
    AUTHOR = {Berline, Nicole and Vergne, Mich\`ele},
     TITLE = {Local {E}uler-{M}aclaurin formula for polytopes},
   JOURNAL = {Mosc. Math. J.},
  FJOURNAL = {Moscow Mathematical Journal},
    VOLUME = {7},
      YEAR = {2007},
    NUMBER = {3},
     PAGES = {355--386, 573},
   MRCLASS = {52B20},
  MRNUMBER = {2343137},
MRREVIEWER = {P.\ McMullen},
       DOI = {10.17323/1609-4514-2007-7-3-355-386},
}

@article {BeckStapledon,
    AUTHOR = {Beck, Matthias and Stapledon, Alan},
     TITLE = {On the log-concavity of {H}ilbert series of {V}eronese
              subrings and {E}hrhart series},
   JOURNAL = {Math. Z.},
  FJOURNAL = {Mathematische Zeitschrift},
    VOLUME = {264},
      YEAR = {2010},
    NUMBER = {1},
     PAGES = {195--207},
   MRCLASS = {05A15 (05E18 13D40 13H10 52B20)},
  MRNUMBER = {2564938},
MRREVIEWER = {Douglas\ A.\ Hanes},
       DOI = {10.1007/s00209-008-0458-7},
}

@article {BrentiWelker,
    AUTHOR = {Brenti, Francesco and Welker, Volkmar},
     TITLE = {The {V}eronese construction for formal power series and graded
              algebras},
   JOURNAL = {Adv. in Appl. Math.},
  FJOURNAL = {Advances in Applied Mathematics},
    VOLUME = {42},
      YEAR = {2009},
    NUMBER = {4},
     PAGES = {545--556},
   MRCLASS = {05E40 (13D40 13F55)},
  MRNUMBER = {2511015},
MRREVIEWER = {Adam\ L.\ Van Tuyl},
       DOI = {10.1016/j.aam.2009.01.001}
}

@article{McMullen_upper_bound,
    AUTHOR = {Seidel, Raimund},
     TITLE = {The upper bound theorem for polytopes: an easy proof of its
              asymptotic version},
   JOURNAL = {Comput. Geom.},
  FJOURNAL = {Computational Geometry. Theory and Applications},
    VOLUME = {5},
      YEAR = {1995},
    NUMBER = {2},
     PAGES = {115--116},
   MRCLASS = {52B11},
  MRNUMBER = {1353291},
MRREVIEWER = {P.\ McMullen},
       DOI = {10.1016/0925-7721(95)00013-Y}
}

\end{document}